\documentclass{article}

\usepackage{fullpage}

\usepackage{graphicx}

\usepackage{amsmath,amsfonts,amssymb,epsfig}

\newtheorem{lemma}{Lemma}
\newtheorem{proposition}{Proposition}

\newtheorem{definition}{Definition}
\newtheorem{remark}{Remark}

\newcommand{\interior}{\operatorname{int}}

\newenvironment{proof}{\noindent {\em Proof .} \hspace{2 mm}}{\vspace{2mm} \vrule height4pt width3pt depth2pt}

\begin{document}

\title{About Transgressive Over-Yielding in the Chemostat} 
\author{D. Dochain$^{1}$, P. de Leenheer$^{2}$ and
  A. Rapaport$^{3}$\footnote{corresponding author}\\
$^{1}$ CESAME, U.C. Louvain, Belgium (e-mail:
Denis.Dochain@uclouvain.be)\\
$^{2}$ Math. Dept., Univ. Florida, USA (e-mail: deleenhe@ufl.edu)\\
$^{3}$ UMR INRA/SupAgro 'MISTEA' and INRA-INRIA associated team
'MODEMIC',\\ Montpellier, France (e-mail: rapaport@supagro.inra.fr)
}
\date{June 20, 2012}

\maketitle

\begin{abstract}
We show that for certain configurations of two chemostats fed in
parallel, the presence of two different species in each tank can
improve the yield of the whole process, compared to the same
configuration having the same species in each volume. This leads to a
(so-called) ``transgressive over-yielding'' due to
spatialization.\\
{\bf Key-words.} Chemostat, steady-state, optimization, over-yielding.
\end{abstract}

\section{Introduction}

The study of biodiversity and environment on the functioning of
ecosystems is one of the most important issues in ecology nowadays (see for instance \cite{RB09}). 

Recent investigations on the functioning of microbial ecosystems has revealed that 
combinations of certain species and resources enhance the functioning of the overall ecosystem, leading to  a ``transgressive over-yielding'' (\cite{SHSL08,LBSJ10}). In those experiments, the environment is modified by altering the composition of the resources and its richness (\cite{TRKKB08}).
In the literature in ecology, the consideration of spatial heterogeneity in the environment (niches, islands...) is often promoted as a leading factor for the explanation and prediction of abundance and distribution of species and resources in ecosystems. Comparatively, the influence of the spatialization on the performances of the functioning of ecosystems has been much more rarely studied (\cite{ERH06}).

From the view point of the industry of biotechnology and waste water treatment, the spatial decoupling of bio-transformation is known to impact significantly the yield conversion, and engineers are looking for design of bio-processes that offer the best performances (see \cite{LT82,HR89,GBBT96,HD05}).
For instance, a series of continuous stirred bioreactors is known to be potentially more efficient than a single one (for precise distributions of the volumes) in \cite{HRT03,HRD04,NS06}.
Recent investigations on the chemostat model have revealed that other simple spatial motifs, such as parallel interconnections of tanks, can also present interesting benefits (\cite{HRG11}).
In this framework, most of the studies are conducted for a single species. Only a few of them investigate the role of biodiversity and spatial heterogeneity on the performances of the ecosystem (\cite{SF79,RHM08}).

In the present work, we show that spatialization can be another factor responsible for a transgressive over-yielding. For this purpose, we consider the usual framework of the chemostat and show that splitting one volume into two tanks in parallel with different dilution rates can enhance the performances at steady state when two different species are present in each vessel, compared to the situations of having the same species present in both volumes. Here, optimal performance is achieved when the value of the nutrient at steady state in the outflow of the system is minimized. This notion is inspired by the scenario of 
waste water treatment processes where the input nutrient to the reactor is the pollutant, and the purpose of the treatment plant is to minimize the pollutant concentration.

We first study the case of a single species and show that two
reactors in parallel instead of one does not bring any benefit, although bypassing one tank may improve the performance in some cases. Then, we give sufficient conditions for the growth curves of two species  that  
guarantee the existence of two tanks configurations that offer a transgressive over-yielding.  

\section{One species}
\subsection{One reactor}
Let a chemostat be operated with constant volumetric  flow rate $Q$, constant volume $V$, and constant input nutrient concentration $S^{in}$. 
We call $D=Q/V$ the dilution rate. In case only 1 species is present, the equations for the nutrient concentration $S(t)$ and the 
species concentration $x(t)$, are:
\begin{eqnarray}
{\dot S}&=&D(S^{in}-S)-\frac{f(S)}{\gamma}x\label{c1}\\
{\dot x}&=&x(f(S)-D)\label{c2}
\end{eqnarray}
where $\gamma\in (0,1]$ is the yield of the conversion of nutrient into species biomass, and $f(S)$ is the per capita growth rate, a smooth, increasing 
function (with $f'(S)>0$ for all $S\geq 0$) with $f(0)=0$. We also assume that $f$ is concave, i.e. $f''(S)\leq 0$ for all $S\geq 0$. Finally, we 
let $D^*=f(S^{in})$.

A typical example satisfying all these conditions would be the Monod function $f(S)=mS/(a+S)$ for positive $a$ and $m$.

The first statement of the following result is well-known (see for instance the textbook \cite{SW95}).
\begin{proposition}\label{f-inverse} Let $x(0)>0$. Every solution $(S(t),x(t))$ of $(\ref{c1})-(\ref{c2})$ satisfies:
$$ 
\lim_{t\rightarrow +\infty}S(t)= \lambda(D),
$$ 
where
\begin{equation}\label{lambda}
\lambda(D)=
f^{-1}(D), \textrm{ if } D\leq D^*,\textrm{ and }
=S^{in}, \textrm{ if } D> D^*
\end{equation}
The (continuous) function $\lambda (D)$ is increasing on $[0,D^*)$ (since $\lambda' > 0$ there) and nondecreasing on $[0,+\infty)$. 
$\lambda(D)$ is convex if $D\in [0,D^*)$ (since $\lambda'' \geq 0$ there), but not if $D\in [D^*,\infty)$.
\end{proposition}
\begin{proof}
That $\lambda(D)$ is continuous, is obvious from the definition, and that it is increasing on $[0,D^*)$ follows from the fact that 
it is the inverse function of the increasing function $f(S)$, $S\in [0,S^{in})$. 
To show the convexity of $\lambda(D)$ for $D\in [0,D^*)$ we implicitly differentiate 
$f^{-1}(f(S))=S$, $S\in [0,S^{in}]$, twice:
$$
(f^{-1})'(f(S))f'(S)=1
$$
$$\Rightarrow (f^{-1})''(f(S))(f'(S))^2 + (f^{-1})'(f(S))f''(S)=0,
$$
and thus that
$$
(f^{-1})''=-\frac{(f^{-1})'f''}{(f')^2}
$$
Since $f'>0$ and hence also $(f^{-1})'>0$, and since $f''\leq 0$, it follows that $(f^{-1})''\geq 0$, hence $f^{-1}$ is convex, and therefore 
so is $\lambda(D)$ as long as $D\in [0,D^*)$.
Nonconvexity of $\lambda (D)$ when $D\in[0,+\infty)$ follows by picking an arbitrary ${\bar D}>D^*$, and letting 
$D^*=p 0+ (1-p){\bar D}$ for an appropriate $p \in (0,1)$. It follows that:
$$
\lambda(p 0 +(1-p){\bar D})=\lambda(D^*) =S^{in}$$
$$> (1-p)S^{in}=(1-p){\bar D}=p\lambda(0)+(1-p)\lambda ({\bar D}),
$$
and thus $\lambda(D)$ is not convex if $D\in[0,+\infty)$.
\end{proof}
\subsection{Two reactors}
Assume that we split the inflow channel to the chemostat in two channels, and the reactor volume in two parts. The first reactor 
is fed by the first channel with volumetric flow rate $\alpha Q$ for some $\alpha \in [0,1]$ and it has volume $rV$ for some $r\in [0,1]$. 
The second reactor is fed by the second channel with volumetric flow rate $(1-\alpha) Q$, and it has volume $(1-r)V$. 
Note that the total reactor volume of both reactors equals the initial
single reactor volume. Note also that the initial inflow channel has
simply been split into two channels. Of course, we assume that for
both species into bioreactors optimal environmental conditions (temperature, pH etc.) are applied.

To determine the asymptotic value of the nutrient concentration 
in each of the two reactors, we use the Proposition from the previous section. This value is completely determined by the dilution rate 
that each reactor experiences. For the first reactor, this dilution rate is 
$$
\frac{\alpha Q}{rV}=\frac{\alpha}{r}D,
$$
whereas it is 
$$
\frac{(1-\alpha)Q}{(1-r)V}=\frac{1-\alpha}{1-r}D
$$
for the second reactor.
Therefore, the asymptotic nutrient values are
$$
\lambda\left(\frac{\alpha }{r}D \right)\textrm{ and } \lambda\left(\frac{1-\alpha}{1-r}D \right),
$$
in the first and second reactor respectively. Assume that we mix the outflows of both reactors (in the same proportions $\alpha$ and $1-\alpha$ as the volumetric flow rates at their respective input channels). The asymptotic value of the nutrient in the mixture is then given by the function
\begin{equation}\label{yield}
F(\alpha ,r):=\alpha\lambda\left(\frac{\alpha }{r}D \right) + (1-\alpha)\lambda\left(\frac{1-\alpha}{1-r}D \right),
\end{equation}
where $(\alpha,r)\in(0,1)^2$. Clearly, this function is smooth in $(0,1)^2$. We would like to extend it to $[0,1]^2$. First we extend it continuously to $(\alpha,r)\neq (0,0)$ and $\neq (1,1)$ by defining
\begin{eqnarray*}
F(\alpha,0)&:=&\alpha S^{in}+(1-\alpha)\lambda \left((1-\alpha)D\right)\textrm{ for } \alpha>0,\\
F(\alpha,1)&:=&\alpha\lambda\left(\alpha D \right)+(1-\alpha)S^{in}\textrm{ for }\alpha<1
\end{eqnarray*}
Finally, we continuously extend it to $[0,1]^2$ by defining:
$$
F(0,0)=F(1,1):=\lambda(D)
$$

An important property of $F$ is that:
$$
F(\alpha, r)=\lambda(D),\textrm{ whenever } \alpha=r.
$$
This is not surprising since all configurations with $\alpha=r$ result in the same dilution rate $D$ for each of the two reactors, which is the same as the dilution rate of the initial single reactor. The asymptotic nutrient concentrations in each reactor will be the same, and equal 
to $\lambda(D)$, which equals the asymptotic nutrient concentration of the single reactor. The mixing of these concentrations with proportions  $\alpha$ and $1-\alpha$, yields $\lambda(D)$. In other words, for these configurations, the splitting of the reactor is merely artificial, and the 
asymptotic outcome is the same as if the reactor would not have been split at all.

An obvious question is to determine those configurations $(\alpha ,r)$ where $F(\alpha,r)$ is minimal, and also to see if these minima  
are lower than the corresponding value $\lambda (D)$ which is obtained in case of a single reactor (or in case of any configuration 
satisfying $\alpha=r$ in view of the remark above)

We first define a convex polygonal region in the unit square $[0,1]^2$:
$$
\begin{array}{lll}
C & = & \{(\alpha,r)\in [0,1]^2\;|\; (\alpha/r)D\leq D^*\\
& & \textrm{ and }\left((1-\alpha)/(1-r)\right)D\leq D^* \}
\end{array}
$$
and we note that $C$ is a nontrivial set (i.e. has nonzero Lebesgue measure) if and only if 
$$
D<D^*,
$$
an assumption we make in the remainder of this paper. This condition on $D$ simply expresses that in case of a single reactor, the 
asymptotic nutrient concentration $\lambda(D)$ will be strictly less than the maximal possible value $S^{in}$.
Next, we also define two triangular regions in $[0,1]^2$:
\begin{eqnarray*}
T_1&=&\{ (\alpha,r)\in [0,1]^2\;|\; (\alpha/r)D> D^*\}\\
T_2&=&\{(\alpha,r)\in [0,1]^2\;|\; \left((1-\alpha)/(1-r)\right)D> D^*) \}
\end{eqnarray*}
and we note that $C$, $T_1$ and $T_2$ are pairwise disjoint, and that they cover $[0,1]^2$.

The following Lemma is usefull for the following.
\begin{lemma}
The function 
\begin{equation}
\delta: (\alpha,r)\in [0,1]\times(0,1] \mapsto \frac{\alpha^{2}}{r}
\end{equation}
is convex on its domain.
\end{lemma}

\begin{proof} One has
$$
\delta_{\alpha} = \frac{2\alpha}{r} \ , \;
\delta_{r} = -\frac{\alpha^{2}}{r^{2}} \ ,
$$
and furthermore
$$
\delta_{\alpha\alpha} = \frac{2}{r} \ , \; 
\delta_{rr} = 2\frac{\alpha^{2}}{r^{3}} \ , \;
\delta_{\alpha r} = - \frac{2\alpha}{r^{2}} \ .
$$
It follows that
$$
\delta_{\alpha\alpha}>0 \ ,\; \delta_{rr}>0 \ , \;
\delta_{\alpha\alpha}\delta_{rr}-\delta_{\alpha r}^{2}=0 \ .
$$
Thus the Hessian of $\delta$ is positive semi-definite.
\end{proof}

We introduce the following function
\begin{equation}
\label{threshold-Sin}
T^{in}(D)=\lambda(D)+D\lambda^{\prime}(D)
\end{equation}
that plays an important role in the following.

The following Proposition shows that the performance of a single species chemostat can only be improved if $S^{in} $ is
small enough, and this is achieved by bypassing the second reactor with a unique, specific flow rate.
\begin{proposition}
The restriction of $F$ to the convex set $C$, $F|_C$, is convex. 
Moreover there holds that
$$ 
\min_{[0,1]^2}F=\left|\begin{array}{l}
\lambda (D)=F(\alpha,\alpha),\; \forall \alpha \in [0,1], \mbox{ for } S_{in}\geq T^{in}(D)\\
\min_{\alpha}F(\alpha,0)<\lambda(D), \mbox{ for } S_{in}< T^{in}(D)
\end{array}\right.
$$
\end{proposition}

\begin{proof}
Let us first show that $F|_C$ is convex. 
We consider the functions $L$ and $R$ defined on $C$:
\begin{equation}
L(\alpha,r)=\alpha\lambda\left(\frac{\alpha}{r}D\right) \ , \; 
R(\alpha,r)=(1-\alpha)\lambda\left(\frac{1-\alpha}{1-r}D\right)
\end{equation}
and show that $L$ and $R$ are convex on $C$.\\

Let $(\alpha_{1},r_{1})$, $(\alpha_{2},r_{2})$ belong to $C$ and $\mu$ be a number in $[0,1]$. One has
\begin{eqnarray*}
\Gamma & = & \mu L(\alpha_{1},r_{1})+(1-\mu)L(\alpha_{2},r_{2})\\
& = & \mu\alpha_{1}\lambda\left(\frac{\alpha_{1}}{r_{1}}D\right)+
(1-\mu)\alpha_{2}\lambda\left(\frac{\alpha_{2}}{r_{2}}D\right)\\
& = & (\mu\alpha_{1}+(1-\mu)\alpha_{2})\left[\frac{\mu\alpha_{1}}{\mu\alpha_{1}+(1-\mu)\alpha_{2}}\lambda\left(\frac{\alpha_{1}}{r_{1}}D\right)\right.\\
& & \hspace{30mm}\left.+
\frac{(1-\mu)\alpha_{2}}{\mu\alpha_{1}+(1-\mu)\alpha_{2}}\lambda\left(\frac{\alpha_{2}}{r_{2}}D\right)\right]
\end{eqnarray*}
By convexity of the function $\lambda$, it follows
\begin{eqnarray*}
\Gamma &\geq & (\mu\alpha_{1}+(1-\mu)\alpha_{2})\lambda\left(
\frac{\mu\alpha_{1}^{2}/r_{1}+(1-\mu)\alpha_{2}^{2}/r_{2}}{\mu\alpha_{1}+(1-\mu)\alpha_{2}}D\right)\\
& = & (\mu\alpha_{1}+(1-\mu)\alpha_{2})\lambda\left(
\frac{\mu\delta(\alpha_{1},r_{1})+(1-\mu)\delta(\alpha_{2},r_{2})}{\mu\alpha_{1}+(1-\mu)\alpha_{2}}D\right)
\end{eqnarray*}
The function $\lambda$ being increasing and $\delta$ being convex (see the former Lemma), one obtains
\begin{eqnarray*}
\Gamma & \geq & (\mu\alpha_{1}+(1-\mu)\alpha_{2}).\\
       &      & \lambda\left(
\frac{\delta(\mu\alpha_{1}+(1-\mu)\alpha_{2},\mu r_{1}+(1-\mu)r_{2})}{\mu\alpha_{1}+(1-\mu)\alpha_{2}}D\right)\\
& = & (\mu\alpha_{1}+(1-\mu)\alpha_{2})\lambda\left(
\frac{\mu\alpha_{1}+(1-\mu)\alpha_{2}}{\mu r_{1}+(1-\mu)r_{2}}D\right)\\
& = & L(\mu\alpha_{1}+(1-\mu)\alpha_{2},\mu r_{1}+(1-\mu)r_{2})
\end{eqnarray*}
Similarily, it can be shown that $R$  is convex on $C$. The function $F$ being the sum of the two functions $L$ and $R$, one concludes that $F$ is also convex on $C$.\\

For $(\alpha,r)\in \interior(C)$, we have that
\begin{eqnarray}
F_\alpha&=& \lambda \left(\frac{\alpha}{r}D \right)+\frac{\alpha D}{r}\lambda'\left(\frac{\alpha}{r}D \right)-\lambda \left(\frac{1-\alpha}{1-r}D \right)\\
 & & \hspace{30mm} -\frac{(1-\alpha)D}{1-r}\lambda' \left(\frac{1-\alpha}{1-r}D \right) \label{F-alfa}\\
F_r&=&-\frac{\alpha^2D}{r^2}\lambda' \left(\frac{\alpha}{r}D \right)+\frac{(1-\alpha)^2D}{(1-r)^2}\lambda'\left(\frac{1-\alpha}{1-r}D \right)\label{F-r}
\end{eqnarray}
We note from $(\ref{F-alfa})$ and $(\ref{F-r})$ that every point on the diagonal where $\alpha=r$ in $\interior(C)$ is a critical point, i.e.:
$$
F_\alpha|_{\alpha=r}=F_r|_{\alpha=r}=0,
$$
and thus, since $F$ is convex in $C$, it follows that
\begin{equation}\label{minimum}
\min_{C}F=\lambda(D)=F(\alpha,\alpha),\;\; \forall \alpha \in [0,1]
\end{equation}
Since $[0,1]^2$ is the disjoint union of $C$, $T_1$ and $T_2$, the proof of the Proposition will be complete if we study the minimisation of $F$ on $T_1$ and $T_2$. Notice that $T_2=\{(\alpha,r)\in [0,1]^2\;|\; (1-\alpha,1-r)\in T_1\}$ and $F(\alpha,r)=F(1-\alpha,1-r)$. Consequently, one has $\min_{T_1}F=\min_{T_2}F$, and we can study the restriction of $F$ to $T_1$ only.
When $(\alpha,r)\in T_1$, then
$$
F(\alpha,r)=\alpha S^{in}+(1-\alpha)\lambda \left( \frac{1-\alpha}{1-r}D \right)
$$ 
Note that $F$ has no critical points in $\interior(T_1)$ because
$$
F_r=\frac{(1-\alpha)^2D}{(1-r)^2}\lambda' \left(\frac{1-\alpha}{1-r}D \right)>0,
$$
and thus the extrema of $F|_{T_1}$ are necessarily located on the boundary of $T_1$. We calculate the values of $F$ on the various 
parts of the boundary of $T_1$, and bound them below:
\begin{eqnarray*}
F(\alpha,0)&=&\alpha S^{in}+(1-\alpha)\lambda ((1-\alpha)D)\\
F(1,r)&=&S^{in}\geq \lambda(D)\\
F(\alpha,\alpha\frac{D}{D^*})&=&\alpha S^{in}+(1-\alpha)\lambda \left(\frac{1-\alpha}{1-\alpha\frac{D}{D^*}}D \right)\\
 & \geq &  \alpha S^{in}+(1-\alpha)\lambda((1-\alpha)D)=F(\alpha,0)
\end{eqnarray*}
where we used that $\lambda(.)$ is non-decreasing. One has
\begin{eqnarray*}
F_\alpha(\alpha,0) & = & S^{in}-\lambda((1-\alpha)D)-D\lambda^{\prime}((1-\alpha)D)(1-\alpha)\\
F_{\alpha\alpha}(\alpha,0) & = & 2D\lambda^{\prime}((1-\alpha)D)+D^{2}\lambda^{\prime\prime}((1-\alpha)D)(1-\alpha)
\end{eqnarray*}
and one concludes that $\alpha \mapsto F(\alpha,0)$ is strictly convex, $\lambda$ being a convex increasing function on the domain $[0,D]$. Furhermore, one has 
\begin{eqnarray*}
F_\alpha(0,0) & = & S^{in}-T^{in}(D)\\
F_\alpha(1,0) & = & S^{in}
\end{eqnarray*}
Consequently, the unique minimum of $F(\alpha,0)$ over $\alpha$ is reached at $\alpha=0$ exactly when $F_\alpha(0,0)\geq 0$, that is to write $S^{in}\geq T^{in}(D)$. Finally, notice that one has 
$F(0,0)=\lambda(D)$ and the proof of the Proposition now follows from $(\ref{minimum})$. 
\end{proof}

\begin{remark} 
\label{rem1}
When $S_{in}<T^{in}(D)$ the best configuration is obtained for the by-pass of one tank with $\alpha$ different than $0$ or $1$. For the Monod function, the threshold on $S^{in}$ has the following expression
$$ 
T^{in}(D)=\frac{Da}{m-D}\left(1+\frac{m}{m-D}\right) \ .
$$
\end{remark}

\section{Two species}
\label{section2}

We consider two species with respective per capita growth rate functions $f_1$  and $f_2$, satisfying the conditions 
of a single species we imposed in the previous section. We assume that there is a unique ${\bar S}\in (0,S^{in})$ such that
\begin{eqnarray*}
f_1(S)&>&f_2(S)\textrm{ if } S\in (0,{\bar S})\\
f_2(S)&>&f_1(S)\textrm{ if } S \in ({\bar S},+\infty)
\end{eqnarray*}
and we let
$$
{\bar D}=f_1({\bar S})=f_2({\bar S}),\;\; D_1^*=f_1(S^{in}),\;\; D_2^*=f_2(S^{in})
$$
We assume as before that
$$
D<\min (D_1^*,D_2^*)
$$

\begin{definition}
We say that a configuration $(\alpha,r)\in (0,1)^2$ corresponds to transgressive overyielding if
\begin{equation}\label{overyield}
\begin{array}{lll}
G(\alpha,r) & := & \displaystyle \alpha \lambda_1 \left(\frac{\alpha}{r}D \right )+(1-\alpha)\lambda_2 \left(\frac{1-\alpha}{1-r}D \right)\\
& < & \min (F_1(\alpha ,r), F_2 (\alpha,r)),
\end{array}
\end{equation}
where $\lambda_i$ is defined as in $(\ref{lambda})$ but with $f=f_i$ and $D^*=D^*_i$, 
and $F_i$ is defined as in $(\ref{yield})$ but with 
$\lambda=\lambda_i$.
\end{definition}

Let us first consider the case where $\alpha=r$. Then
\begin{eqnarray*}
G(\alpha,\alpha)&=&\alpha \lambda_1(D)+(1-\alpha)\lambda_2(D)\\
F_1(\alpha,\alpha)&=&\alpha \lambda_1(D)+(1-\alpha)\lambda_1(D)=\lambda_1(D)\\
F_2(\alpha,\alpha)&=&\alpha \lambda_2(D)+(1-\alpha)\lambda_2(D)=\lambda_2(D)
\end{eqnarray*}

Consequently, since any convex combination of two non-negative numbers cannot be smaller than the smallest of these two numbers, 
it follows that
\begin{equation}\label{special}
\begin{array}{lll}
G(\alpha,\alpha) & = & \alpha \lambda_1(D)+(1-\alpha)\lambda_2(D)\\
& \geq & \min(\lambda_1(D),\lambda_2(D))
 = \min(F_1(\alpha,\alpha),F_2(\alpha,\alpha)), 
\end{array}
\end{equation}
which implies that no configuration with $\alpha=r$ can correspond to transgressive overyielding. Moreover, if $\lambda_1(D)\neq \lambda_2(D)$, then the inequality in $(\ref{special})$ is strict for all $\alpha \in (0,1)$. By continuity of $G,F_1$ and $F_2$ this inequality remains strict 
for all configurations $(\alpha,r)$ near such configurations where $\alpha=r$ and $\alpha\in (0,1)$, implying that we will not find 
configurations corresponding to transgressive overyielding nearby. On the other hand, if $\lambda_1(D)=\lambda_2(D)={\bar S}$, or equivalently if 
$D={\bar D}$, then the inequality in $(\ref{special})$ reduces to an equality for all $\alpha \in (0,1)$. In this case we show that there always exist 
configurations nearby corresponding to transgressive overyielding.

The two next Propositions show that there always exist configurations of two parallel chemostats with only one
of the species in each reactor volume that perform better than in the case in which for the same
configuration, only one of the species is present in both reactors, regardless of the selected
species.

\begin{proposition}
\label{overyield-config}
Let $D={\bar D}$ and fix $\alpha \in (0,1)$. Then there always exist configurations $({\tilde \alpha},{\tilde r})$ near $(\alpha,\alpha)$ that 
correspond to transgressive overyielding.
\end{proposition} 
\begin{proof}
We start with the simple observation that for any $({\tilde \alpha},{\tilde r})\in (0,1)^2$:
$$
\frac{{\tilde \alpha}}{{\tilde r}}\lesseqqgtr1 \Longleftrightarrow \frac{1-{\tilde \alpha}}{1-{\tilde r}}\gtreqqless1
$$
Therefore, for fixed $\alpha\in (0,1)$, we can always find $({\tilde \alpha},{\tilde r})\in (0,1)^2$ near $(\alpha,\alpha)$ such that
$$
\frac{{\tilde \alpha}}{{\tilde r}}{\bar D}<{\bar D}<\frac{1-{\tilde \alpha}}{1-{\tilde r}}{\bar D}
$$
Moreover, since $D={\bar D}< \min (D_1^*,D_2^*)$, we may assume w.l.o.g. that 
$$
\frac{1-{\tilde \alpha}}{1-{\tilde r}}{\bar D}<\min(D_1^*,D_2^*)
$$
The properties of $f_1$ and $f_2$ imply that:
\begin{eqnarray*}
0 & < & \lambda_1 \left(\frac{{\tilde \alpha}}{{\tilde r}}{\bar D} \right)<\lambda_2 \left(\frac{{\tilde \alpha}}{{\tilde r}}{\bar D} \right)<{\bar S}\\
& < & \lambda_2\left(\frac{1-{\tilde \alpha}}{1-{\tilde r}}{\bar D} \right)<\lambda_1\left(\frac{1-{\tilde \alpha}}{1-{\tilde r}}{\bar D}\right)
<S^{in}
\end{eqnarray*}
and thus in particular that:
\begin{eqnarray*}
\label{order}
0 & < & \lambda_1 \left(\frac{{\tilde \alpha}}{{\tilde r}}{\bar D} \right)<\lambda_2 \left(\frac{{\tilde \alpha}}{{\tilde r}}{\bar D} \right)\\
& < & \lambda_2\left(\frac{1-{\tilde \alpha}}{1-{\tilde r}}{\bar D} \right)<\lambda_1\left(\frac{1-{\tilde \alpha}}{1-{\tilde r}}{\bar D} \right)
\end{eqnarray*}
There holds that
\begin{eqnarray*}
G({\tilde \alpha},{\tilde r})&=&{\tilde \alpha} \lambda_1\left(\frac{{\tilde \alpha}}{{\tilde r}}{\bar D}\right)+(1-{\tilde \alpha})\lambda_2\left(\frac{1-{\tilde \alpha}}{1-{\tilde r}}{\bar D} \right)\\
F_1({\tilde \alpha},{\tilde r})&=&{\tilde \alpha} \lambda_1\left(\frac{{\tilde \alpha}}{{\tilde r}}{\bar D} \right)+(1-{\tilde \alpha})\lambda_1\left(\frac{1-{\tilde \alpha}}{1-{\tilde r}}{\bar D} \right)\\
F_2({\tilde \alpha},{\tilde r})&=&{\tilde \alpha} \lambda_2\left(\frac{{\tilde \alpha}}{{\tilde r}}{\bar D} \right)+(1-{\tilde \alpha})\lambda_2\left(\frac{1-{\tilde \alpha}}{1-{\tilde r}}{\bar D} \right)
\end{eqnarray*}
Noting that each of $G({\tilde \alpha},{\tilde r}),F_1({\tilde \alpha},{\tilde r})$ and $F_2({\tilde \alpha},{\tilde r})$ is a convex combination 
with the same weights ${\tilde \alpha}$ and $1-{\tilde \alpha}$
of a pair of distinct positive numbers from a collection of four positive numbers that are ordered as indicated by $(\ref{order})$, it follows that:
$$
G({\tilde \alpha},{\tilde r})<\min \left(F_1({\tilde \alpha},{\tilde r}),F_2({\tilde \alpha},{\tilde r}) \right),
$$
and thus the configuration $({\tilde \alpha},{\tilde r})$ corresponds to transgressive overyielding.
\end{proof}

The case where $D\neq {\bar D}$, can be handled similarly:
\begin{proposition}
Let $D\neq {\bar D}$. Then there exist $(\alpha,r)\in (0,1)^2$ such that
\begin{equation}\label{choice}
\frac{\alpha}{r}D<{\bar D}<\frac{1-\alpha}{1-r}D<\min(D_1^*,D_2^*)
\end{equation}
and this configuration $(\alpha,r)$ corresponds to transgressive overyielding.
\end{proposition}
\begin{proof}
Using the fact that ${\bar D},D<\min(D_1^*,D_2^*)$, 
it is not hard to show that there always exist $(\alpha,r)\in (0,1)^2$ such that $(\ref{choice})$ holds. Indeed, the half spaces
$\{(\alpha,r)|\alpha/r<D/{\bar D}\}$ and $\{(\alpha,r)|(1-\alpha)/(1-r)<\min(D_1^*,D_2^*)/D\}$ are easily seen to intersect in the open unit 
square $(0,1)^2$.
From this follows that $(\ref{order})$ holds with $({\tilde \alpha},{\tilde r})$ replaced by $(\alpha,r)$, and then the rest of the 
proof of the previous Proposition can be carried out with $({\tilde \alpha},{\tilde r})$ replaced by $(\alpha,r)$. In conclusion, 
the configuration $(\alpha,r)$ corresponds to transgressive overyielding.
\end{proof}

\section{Numerical examples}
We consider two species whose graphs of growth functions $f_{i}$ ($i=1, 2$) intersect away from zero~:
$$
f_{1}(s)=\frac{2s}{3+s} \, , \qquad f_{2}(s)=\frac{s}{0.3+s}
$$
One can easily check that $f_{1}(S^{\star})=f_{2}(S^{\star})$ for $S^{\star}=2.4$. For the simulations, we have chosen $D=0.9$ that is slightly above $f_{1}(S^{\star})=f_{2}(S^{\star})= 0.88$ (see Figure \ref{fig1}).
\begin{figure}[h]%[htbp]
\begin{center}
\includegraphics[scale=0.6]{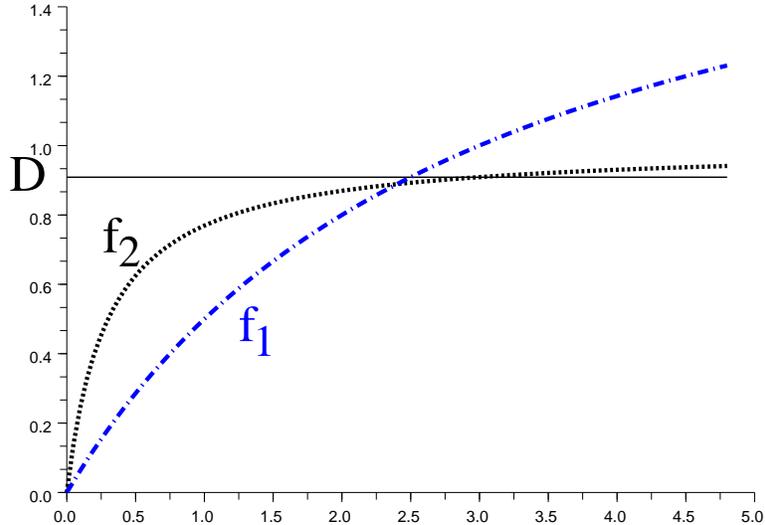}
\caption{Growth curves and dilution rate $D$} \label{fig1}
\end{center}
\end{figure}
For this value of the dilution rate $D$, the break-even concentrations of the two species are
$$ 
\lambda_{1}(D)= 2.499 \, , \qquad \lambda_{2}(D)= 2.993
$$ 
The thresholds given by formula (\ref{threshold-Sin}) are
$$ 
T_{1}^{in}(D)= 7.08 \ , \; T_{2}^{in}(D)= 35.84
$$
We have studied numerically three values of $S^{in}$, larger than both thresholds, in between or lower. The best values for the bypass configurations are reported on the following table
\begin{table}[h]%[htbp]
\begin{center} 
\begin{tabular}{c|c|c|c|}
$S^{in}$                       & 40 & 10 & 5\\
\hline
$\min_{\alpha}F_{1}(\alpha,0)$ & $2.50 \, (=\lambda_{1}(D))$ & $2.50 \, (=\lambda_{1}(D))$ & 2.35\\
$\min_{\alpha}F_{2}(\alpha,0)$ & $2.99 \, (=\lambda_{2}(D))$ & 2.21 & 1.61\\[2mm]
\end{tabular}
\caption{\label{table}Best values of by-pass configurations.}
\end{center}
\end{table}
For $S^{in}=40$, the best configuration is obtained for $\alpha=0.41$ and $r=0.36$ that achieves the minimum of $G$ equal to $0.234$, which represents a gain of $12\%$ compared to having the best species (the first one) in a single tank.
Accordingly to the Table \ref{table}, the best configurations for $S^{in}=10$ and $S^{in}=5$ are obtained for the bypass of the first tank, using the second species only in the second tank (although this species has the worst break-even concentration). Nevertheless, we show on Figures \ref{fig2} 
and \ref{fig3} the benefit of the parallel configuration with non-null volumes for different distributions of the volumes given by the parameter $r$.

One can check that in any case the minimum of $G(\cdot,r)$ over $\alpha$ is lower than $\min(\lambda_{1}(D),\lambda_{2}(D))$ and is reached for a value of $\alpha$ close to $r$.

\begin{figure}[htbp]
\begin{center}
\includegraphics[scale=0.9]{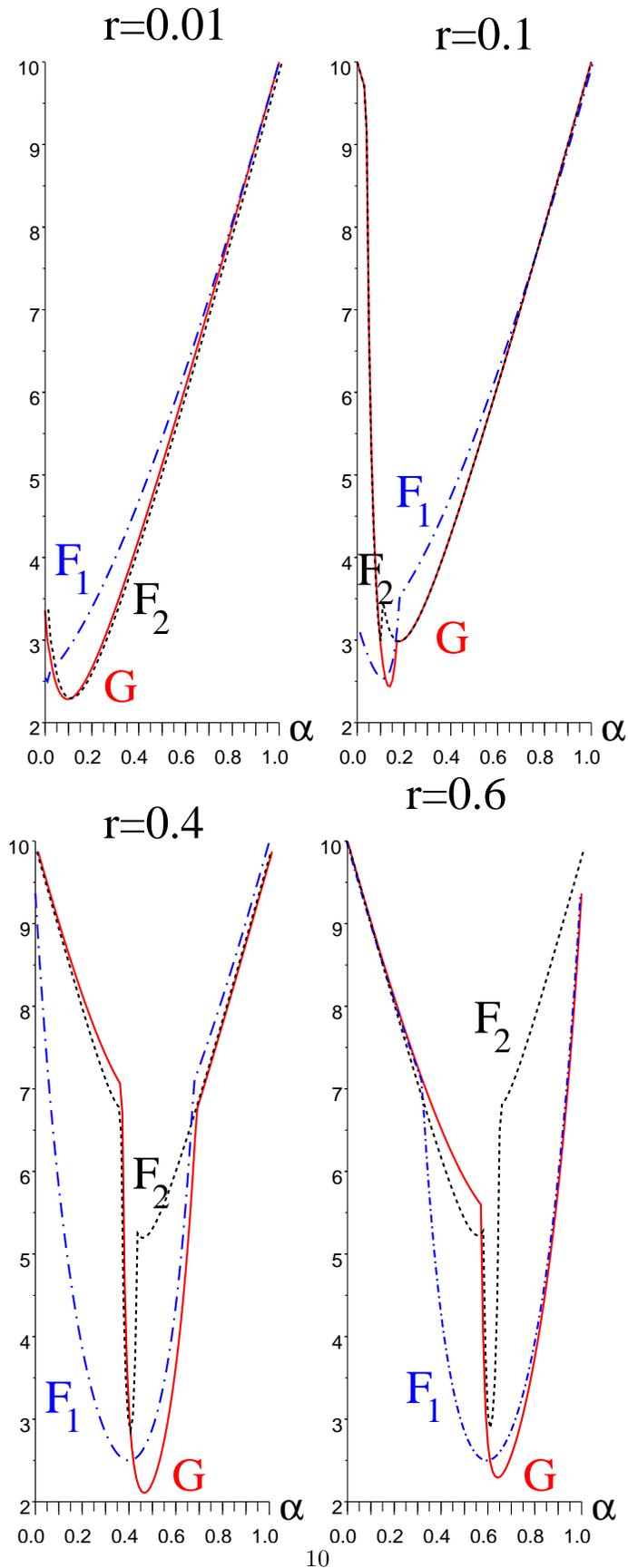}
\caption{Graphs of $F_{i}(\cdot,r)$ and  $G(\cdot,r)$ for $S^{in}=10$ \label{fig2}}
\end{center}
\end{figure} 
\begin{figure}[htbp]
\begin{center}
\includegraphics[scale=0.9]{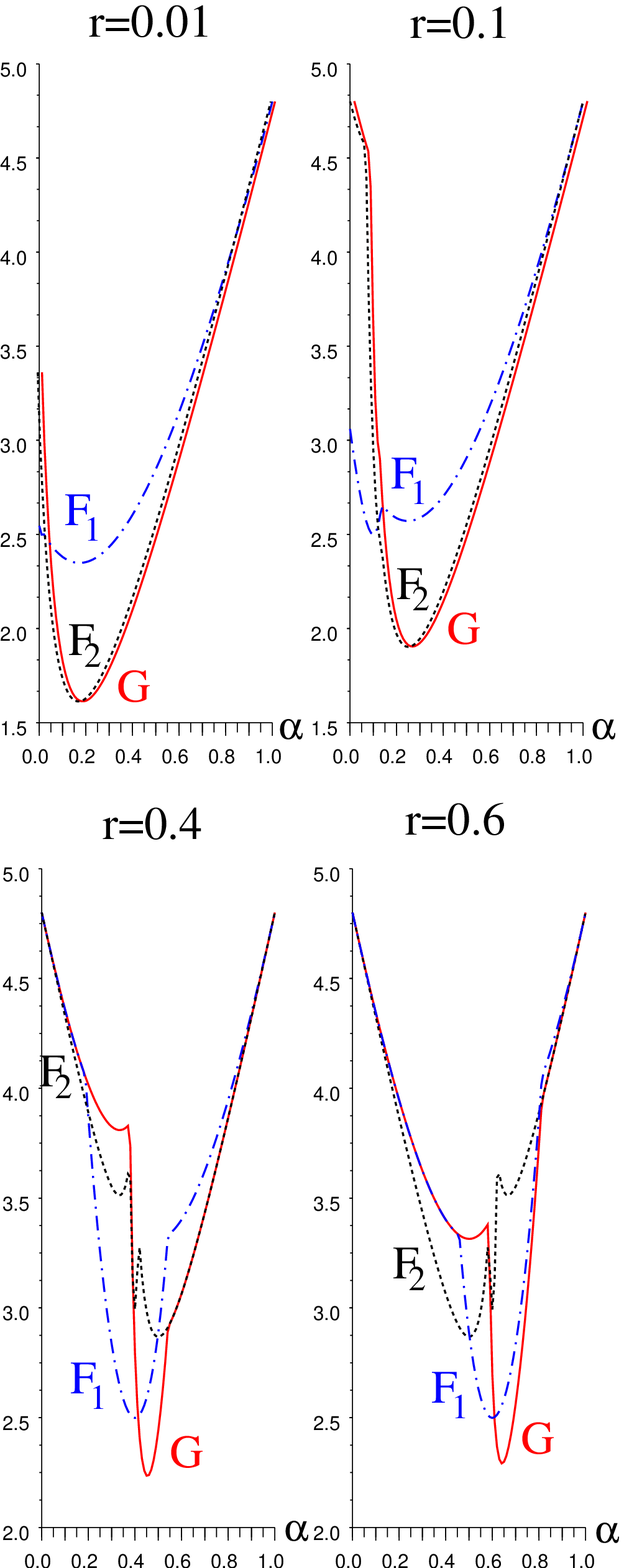}
\caption{Graphs of $F_{i}(\cdot,r)$ and  $G(\cdot,r)$ for $S^{in}=5$ \label{fig3}}
\end{center}
\end{figure}

\section{Conclusion}

Our main contribution is to show that whatever are monotonic growh
rate functions of two species, whose graphs have an intersection away
from the origin, there always exist configurations of two parallel
chemostats with only one of the species in each reactor volume that
perform better than in the case in which for the same configuration,
only one of the species is present in both reactors, regardless of the
selected species. We beleive that this result is of interest for
potential applications in bio-industry. Further investigations could
be conducted to determine more precisely what gain could be
expected. Our simulations show that this gain could be significant.\\

\noindent  {\bf Acknowledgments} 
This research was initiated at INRA in Montpellier during a visit of DD and PDL there. They wish to express their gratitude for the opportunity.
PDL thanks VLAC (Vlaams Academisch Centrum) for hosting him as a research fellow during a sabbatical leave from the University of Florida. PDL also gratefully acknowledges support received from the University of Florida through a Faculty Enhancement Opportunity, and the Universit{\'e} Catholique de Louvain-la-Neuve for providing him with a visiting professorship in the fall of 2011. 

\bibliographystyle{plain}
\bibliography{overyield} 

\end{document}